\documentclass[preprint,12pt]{elsarticle}

\usepackage{amsmath,amssymb,amsthm,amsfonts}
\usepackage{enumerate}
\usepackage{graphics}
\usepackage{graphicx}
\usepackage{url,color}

\newcommand{\R}{\mathbb{R}}
\newcommand{\IR}{\mathbb{IR}}
\newcommand{\inum}[1]{\mathbf{#1}}

\newtheorem{theorem}{Theorem}
\newtheorem{proposition}{Proposition}
\newtheorem{corollary}{Corollary}
\newtheorem{definition}{Definition}
\newtheorem{example}{Example}
\newtheorem{algorithm}{Algorithm}

\begin{document}
\begin{frontmatter}

\title{Necessary and sufficient conditions for regularity of
interval parametric matrices}

\author[ep]{Evgenija D.\ Popova}
\ead{epopova@math.bas.bg}

\address[ep]{Institute of Mathematics and Informatics, Bulgarian Academy of Sciences \\ Acad. G.~Bonchev Str., block 8, 1113 Sofia, Bulgaria.}

\begin{abstract}
Matrix regularity is a key to various problems in applied
mathematics. The sufficient conditions, usually used for checking
regularity of interval parametric matrices, fail in case of large
parameter intervals. We present necessary and sufficient conditions
for regularity of interval parametric matrices in terms of boundary
parametric hypersurfaces, parametric solution sets, determinants,
real spectral radiuses. The initial $n$-dimensional problem
involving $K$ interval parameters is replaced by numerous problems
involving $1\leq t\leq\min\{n-1,K\}$ interval parameters, in
particular $t=1$ is most attractive. The advantages of the proposed
methodology are discussed along with its application for finding the
interval hull solution to interval parametric linear system and for
determining the regularity radius of an interval parametric matrix.
\end{abstract}

\end{frontmatter}

\section{Introduction}
Contrary to the nonparametric case \cite{Rohn09-40nsc}, necessary
and sufficient conditions for regularity of interval parametric
matrices have not been studied in much details. Regularity of the
latter matrices has been studied mainly via some sufficient
conditions \cite{Rump94}, \cite{Popova04}, \cite{Popova19} as part
of the methods for bounding the solution set of an interval
parametric linear system. Checking several properties of interval
matrices, cf. \cite{KreiLakRohnKa}, like positive definiteness,
$P$-property, stability,  can be reduced to checking regularity.
These properties have various other useful applications
\cite{Mansour}, \cite{Jaulin}, \cite{KolPatr}. Some matrix
properties are studied for interval parametric matrices. Stability
of symmetric interval matrices is investigated in
\cite{Rohn'SIMAX94}. Stability in the general case of parameter
dependencies is investigated in \cite{Kolev'CTA12}. In
\cite{Hla-PosDef} a sufficient condition for checking strong
positive definiteness of an interval parametric matrix employs
regularity of the matrix. In all works, regularity of an interval
parametric matrix is estimated by a sufficient condition which may
fail in case of large parameter intervals.

In \cite{Rohn09-40nsc} J. Rohn mentions that ``{\em regularity of
interval matrices is worth further study}''. This is done in the
present work where we present some necessary and sufficient
conditions for regularity of interval parametric matrices. This
property is formulated in several equivalent forms: boundary
parametric hypersurfaces, parametric solution sets, determinants,
real spectral radiuses. A key approach to the presented conditions
is a transformation of the initial problem depending on $K$ interval
parameters into a set of $K2^{K-1}$ problems  depending on only one
interval parameter. This transformation is  based on a set of
one-parameter hypersurfaces, which contains the boundary of the
solution set of an interval parametric linear system.

The article starts by a Notation section and some basic definitions
for interval parametric matrices. The methodology and the necessary
and sufficient conditions are presented in Section \ref{nsc}.
Sections \ref{Hull} and \ref{regRad} discuss important applications
of the regularity: (i) computing the exact interval hull of a
parametric united solution set; (ii) radius of regularity of an
interval parametric matrix. Numerical examples illustrate the
methodology and its applications. The article ends by a Conclusion
section which discusses the advantages of the presented necessary
and sufficient conditions.

\section{Notation}
Denote by $\R^{m\times n}$  the set of real $m\times n$ matrices.
Vectors are considered as one-column matrices. The inequalities are
understood componentwise. The identity matrix of appropriate
dimension is denoted by $I$. The componentwise Hadamard product is
denoted by $\circ$. The vector of all ones is $e=(1,\ldots,1)^\top$.
The $n$-dimensional discrete cube $Q_n := \{y\in\R^n \mid |y|=e\}$
is the set of all $\pm 1$-vectors in $\R^n$ and its cardinality is
$2^n$. For each $y\in\R^n$ we denote the diagonal matrix with
diagonal vector $y$ and zero off-diagonal elements by
$$ D_y = {\rm diag}(y_1,\ldots,y_n).$$
We denote the determinant of a matrix $A\in\R^{n\times n}$ by
$\det(A)$. The spectral radius of a matrix $A\in\R^{n\times n}$ is
denoted by $\rho(A)$. The maximal magnitude real spectral radius
(eigenvalue) of $A$ is denoted by
$$\rho_0(A) :=
\max\{|\lambda| \mid Ax=\lambda x \text{ for some } x\neq 0, \lambda
\text{ real}\}$$ and we set $\rho_0(A)=0$ if  $A$ has no real
eigenvalue.

A real compact interval is $\inum{a} = [a^-, a^+] := \{a\in\R \mid
a^-\leq a\leq a^+\}$ and $\IR^{m\times n}$ denotes the set of
interval $m\times n$ matrices. For $\inum{a} = [a^-, a^+]$, define
its mid-point $\check{a}:= (a^- + a^+)/2$,  the radius $\hat{a} :=
(a^+ - a^-)/2$.  These functions are applied to interval vectors and
matrices componentwise.

Let $A(p)$ be a real $m\times n$ parametric matrix whose elements
$a_{ij}(p)$, $i=1,\ldots,m$, $j=1,\ldots,n$, are given functions of
a number of real parameters $p=(p_1,\ldots,p_K)^\top$.

\begin{definition}\label{def1}
For given real functional dependencies
$A(p)=\left(a_{ij}(p)\right)$, $i=1,\ldots,m$, $j=1,\ldots,n$, and a
given interval vector $\inum{p}\in\IR^K$, such that $\hat{p}_k
> 0$ for each $1\leq k\leq K$, the following set of real matrices
\begin{equation}\label{pMatL}
\left\{A(p), \inum{p}\right\} := \left\{A(p) \mid \exists
p\in\inum{p}\right\}
\end{equation}
is called an $m\times n$  {\em interval parametric matrix}.
\end{definition}
An interval parametric matrix is not an interval matrix and
$\left\{A(p), \inum{p}\right\}$ is a short notation for a family of
real matrices specified by the functional dependencies $A(p)$ and
the parameter intervals $\inum{p}$. The couple $A(p)$,
$p\in\inum{p}$, will be used as an equivalent notation for an
interval parametric matrix.

In this work we consider parameter dependencies defined by
affine-linear functions. Namely,
\begin{equation}\label{affForm}
A(p) = A_0 +\sum_{k=1}^K p_kA_k, \qquad
p\in\inum{p}\in\IR^K,
\end{equation}
with prescribed numerical matrices $A_k$, $k=0,\ldots,K$ and the
parameters $p=(p_1,\ldots$, $p_K)^\top$ are considered to be
uncertain and varying within given non-degenerate\footnote{An
interval $\inum{a}=[a^-, a^+]$ is degenerate if $a^-=a^+$.}
intervals $\inum{p}=(\inum{p}_1,\ldots,\inum{p}_K)^\top$. Nonlinear
dependencies between interval valued parameters in interval
parametric matrices are usually linearized to the form
(\ref{affForm}) and methods for the latter are applied.

Nonparametric interval matrices $\inum{A}\in\IR^{m\times n}$ can be
considered as a special class of parametric interval matrices.
Namely, $\inum{A}\in\IR^{m\times n}$ can be considered as a
parametric interval matrix involving $m\times n$ interval parameters
$a_{ij}\in\inum{a}_{ij}$, $1\leq i\leq m$, $1\leq j\leq n$. Thus,
the following Definition \ref{defReg} comprises both the parametric
and nonparametric interval matrices. In this article we consider
square parametric matrices.

\begin{definition}\label{defReg}
A square interval parametric  matrix $A(p)$, $p\in\inum{p}$, is
called {\em regular} if $A(p)$ is regular for each $p\in\inum{p}$.
\end{definition}
$\{A(p), \inum{p}\}$ is said singular if Definition \ref{defReg} is
not satisfied, i.e., if $A(p)$ is singular for some $p\in\inum{p}$.

\section{Necessary and sufficient conditions for regularity}  \label{nsc}
Let ${\cal K}=\{1,\ldots,K\}$. Define ${\cal K}(m)$ as the set of
all possible subsets of ${\cal K}$ having $m=\min\{n-1,K\}$ elements
$$
{\cal K}(m):=\{q=\{i_1,\ldots,i_m\} \mid q\subseteq{\cal K}\}.
$$
For a vector $p=(p_1,\ldots,p_K)^\top\in\R^K$ and a
$q=\{i_1,\ldots,i_m\}$, $m<K$, define $\tilde{q}:= {\cal K}\setminus
q$ and two vectors $p_q\in\R^m$, $p_{\tilde{q}}\in\R^{K-m}$ by
\begin{eqnarray*}
p_q &:=& (p_{i_1},\ldots, p_{i_m})^\top\\
p_{\tilde{q}} &:=& (p_{i_{m+1}},\ldots, p_{i_K})^\top.
\end{eqnarray*}

For completeness of the exposition, we first present \cite[Theorem
4.2]{PopKraBIT08}, proven here under a more general condition. Let
$$ \Sigma\left(A(p),
b(p), \inum{p}\right) := \left\{x\in\R^n\mid A(p)x=b(p) \text{ for
some } p\in\inum{p}\right\}
$$
be the parametric united solution set of an interval parametric
system $A(p)x=b(p)$, $p\in\inum{p}\in\IR^K$, involving affine-linear
dependencies, and $\partial \Sigma\left(A(p), b(p), \inum{p}\right)$
denote the boundary of the solution set.
\begin{theorem}\label{BIT08}
If $A(\check{p})$ is nonsingular, then
$$\partial\Sigma\left(A(p),
b(p), \inum{p}\right)\subseteq \bigcup_{q\in{\cal
K}(m)}\bigcup_{\varepsilon\in Q_{K-m}}
\left.x\left(p_q,\varepsilon\circ\hat{p}_{\tilde{q}}\right)\right|_{[-\hat{p}_q,
\hat{p}_q]}\subseteq \Sigma\left(A(p), b, \inum{p}\right),
$$
where the parametric hypersurfaces
$$x\left(p_q,\varepsilon\circ\hat{p}_{\tilde{q}}\right) = 
\left(A(\check{p})+\sum_{i\in q}p_iA_i + \sum_{j\in {\cal
K}\setminus
q}\varepsilon_j\hat{p}_jA_j\right)^{-1}\left(b(\check{p})+\sum_{i\in
q}p_ib_i + \sum_{j\in {\cal K}\setminus
q}\varepsilon_j\hat{p}_jb_j\right)
$$
are restricted to the interval vector
$[-\hat{p}_q, \hat{p}_q]$.
\end{theorem}
\begin{proof} Nonsingularity of
$\check{A}$ implies {\rm ker}$(A(p))=\{0\}$, where the kernel (or
null space) of a matrix $A(p)\in \R^{m\times m}$ is
$$ {\rm ker}
(A(p)):= \{x\in\R^m \mid A(p)x = 0\}.$$ Therefore the symbolic
matrix $A(p)$ is invertible and $x(p)=\left(A(p)\right)^{-1}b(p)$
has explicit representation. Then the proof goes the same way as in
\cite[Theorem 4.1 and Theorem 4.2]{PopKraBIT08}.
\end{proof}

\begin{theorem}\label{nscG0}
For  an $n\times n$ interval parametric matrix $A(p)$,
$p\in\inum{p}$, the following conditions are equivalent:
\begin{itemize}
\item[(i)] $\{A(p), \inum{p}\}$ is regular,

\item[(a)] each interval parametric matrix $\{A(p_q, \varepsilon\circ \hat{p}_{\tilde{q}}), [-\hat{p}_q, \hat{p}_q]\}$,
$q\in{\cal K}(m)$, $\varepsilon\in Q_{K-m}$, of the form
\begin{multline}\label{A(p_q)}
A(p_q, \varepsilon\circ \hat{p}_{\tilde{q}}) = \\ A(\check{p}) +
\sum_{i\in q}p_iA_i + \sum_{j\in \tilde{q}}
\varepsilon_j\hat{p}_jA_j, \quad p_q\in [-\hat{p}_q, \hat{p}_q], \;
\tilde{q} = {\cal K}\setminus q,
\end{multline}
is regular,

\item[(b)] for a vector $b(p)\in\R^n$
and for each $q\in{\cal K}(m)$, $\varepsilon\in Q_{K-m}$, defining
$A(p_q,\varepsilon\circ\hat{p}_{\tilde{q}})$ in (\ref{A(p_q)}) and
$b(p_q,\varepsilon\circ\hat{p}_{\tilde{q}})= b(\check{p}) +
\sum_{i\in q}p_ib_i + \sum_{j\in \tilde{q}}
\varepsilon_j\hat{p}_jb_j$, the solution set
\begin{multline*}
\Sigma\left(A(p_q,\varepsilon\circ\hat{p}_{\tilde{q}}),
b(p_q,\varepsilon\circ\hat{p}_{\tilde{q}}), [-\hat{p}_q,
\hat{p}_q]\right) := \\
\{A(p_q,\varepsilon\circ\hat{p}_{\tilde{q}})x=b(p_q,\varepsilon\circ\hat{p}_{\tilde{q}})
\mid \exists p_q\in [-\hat{p}_q, \hat{p}_q]\},
\end{multline*}
is bounded.
\end{itemize}
\end{theorem}
\begin{proof}
(i)$\Rightarrow$(a) by Definition \ref{defReg}.

(a)$\Leftrightarrow$(b) is obvious.

(b)$\Rightarrow$(i) Proof  by contradiction: assume that $\{A(p),
\inum{p}\}$ is singular. This implies that for any $b(p)\in\R^n$,
some boundary hypersurfaces defining the boundary
$\partial\Sigma\left(A(p), b(p), \inum{p}\right)$ are unbounded. On
the other hand, (a) implies nonsingularity of $\check{A}$ and, by
Theorem \ref{BIT08},
$$\partial\Sigma\left(A(p),
b(p), \inum{p}\right)\subseteq \bigcup_{q\in{\cal
K}(m)}\bigcup_{\varepsilon\in Q_{K-m}}
\left.x\left(p_q,\varepsilon\circ\hat{p}_{\tilde{q}}\right)\right|_{[-\hat{p}_q,
\hat{p}_q]}.
$$
By (b), the set in the right-hand side above contains only bounded
restricted parametric hypersurfaces. Thus, we have a contradiction
between the assumption and (b).
\end{proof}

If  $K\geq n$, Theorem \ref{nscG0} reduces the general regularity
problem to $\binom{K}{n-1}2^{K-n+1}$ regularity problems involving
$n-1$ interval parameters. The following theorem presents equivalent
necessary and sufficient conditions for regularity of an interval
parametric matrix by $K2^{K-1}$ parametric problems involving only
one interval parameter $q\in{\cal K}(1)$. It should be noted,
however, that with obvious modifications Theorem \ref{nscGen} below
holds true and can be applied for any $t$, $1\leq t\leq m$, and
$q\in{\cal K}(t)$.

\begin{theorem}\label{nscGen}
For  an $n\times n$ interval parametric matrix $A(p)$,
$p\in\inum{p}\in\IR^K$, the following conditions are equivalent:
\begin{itemize}
\item[(i)] $\{A(p), \inum{p}\}$ is regular,

\item[(ii)] each interval parametric matrix $\{A(p_k,\varepsilon\circ\hat{p}_{\tilde{q}}), [-\hat{p}_k, \hat{p}_k]\}$, $k\in{\cal K}$,
$\varepsilon\in Q_{K-1}$, $\tilde{q}={\cal K}\setminus\{k\}$, of the
form
\begin{equation}\label{A(p_k)}
A(p_k, \varepsilon\circ\hat{p}_{\tilde{q}}) = A(\check{p}) + p_kA_k
+ \sum_{i=1, i\neq k}^K
\varepsilon_i\hat{p}_iA_i, \qquad 
p_k\in [-\hat{p}_k, \hat{p}_k],
\end{equation}
is regular,

\item[(iii)] for a vector $b(p)\in\R^n$
and each $k\in{\cal K}$, $\varepsilon\in Q_{K-1}$, the solution set
\begin{multline*}
\Sigma\left(A(p_k,\varepsilon\circ\hat{p}_{\tilde{q}}),
b(p_k,\varepsilon\circ\hat{p}_{\tilde{q}}), [-\hat{p}_k,
\hat{p}_k]\right) := \\
\left\{\left(A(\check{p})+p_kA_k+\sum_{i=1, i\neq k}^K
\varepsilon_i\hat{p}_iA_i\right)x =
b(p_k,\varepsilon\circ\hat{p}_{\tilde{q}}) \mid \exists p_k\in
[-\hat{p}_k, \hat{p}_k]\right\}
\end{multline*}
is bounded.

\item[(iv)]
each numerical matrix of the form
$$A(\varepsilon^{(k)}) = A(\check{p}) - \sum_{i=1}^K
\varepsilon_i^{(k)}\hat{p}_iA_i,
$$
where $0\leq |\varepsilon_k^{(k)}|\leq 1$,
$\varepsilon_i^{(k)}\in\{-1,1\}$ for $i=1,\ldots,K$, $i\neq k$, and
$k\in{\cal K}$, is nonsingular,

\item[(v)]
$\det\left(A(p_k, \varepsilon\circ\hat{p}_{\tilde{q}})\right) \neq
0$ for each $p_k\in [-\hat{p}_k, \hat{p}_k]$, $k\in{\cal K}$,
$\varepsilon\in Q_{K-1}$, where $A(p_k,
\varepsilon\circ\hat{p}_{\tilde{q}})$ is defined in (\ref{A(p_k)}),

\item[(vi)]  $\check{A}$ is nonsingular and
$$
\rho_0\left(B(\varepsilon_k, \tilde{\varepsilon})\right)<1, \qquad
B(\varepsilon_k, \tilde{\varepsilon}):=
\varepsilon_k\hat{p}_k\check{A}^{-1}A_k + \sum_{i=1, i\neq k}^K
\tilde{\varepsilon}_i\hat{p}_i\check{A}^{-1}A_i,
$$
for each $k\in{\cal K}$, $0\leq |\varepsilon_k|\leq 1$,
$\tilde{\varepsilon}\in Q_{K-1}$,

\end{itemize}
\end{theorem}
\begin{proof}
(i)$\Rightarrow$(ii) by Definition \ref{defReg}.

(ii)$\Leftrightarrow$(iii) is obvious.

The proof (iii)$\Rightarrow$(i) goes  similarly to the proof
(b)$\Rightarrow$(i) in Theorem \ref{nscG0}, since each two
components of $x\left(p_k,\varepsilon\circ\hat{p}_{\tilde{q}}\right)
= A^{-1}(p_k,\varepsilon\circ\hat{p}_{\tilde{q}})
b(p_k,\varepsilon\circ\hat{p}_{\tilde{q}})$ restricted to
$[-\hat{p}_k, \hat{p}_k]$ present $2$-dimensional projections of
either some boundary hypersurfaces of $\Sigma\left(A(p), b(p),
\inum{p}\right)$ or a subset of the solution set.

The equivalence between (ii), (iv) and (v) is obvious.

(i)$\Rightarrow$(vi) by contradiction. Assume that
$$B(\varepsilon_k, \tilde{\varepsilon})x =
\lambda x
$$
for some $x\neq 0$, $|\lambda|\geq 1$, $k=1,\ldots,K$, $0\leq
|\varepsilon_k|\leq 1$, $\tilde{\varepsilon}\in Q_{K-1}$. Then
$$\left(I - \frac{1}{\lambda}B(\varepsilon_k, \tilde{\varepsilon})\right)x = 0.$$
Hence $I -\frac{1}{\lambda}B(\varepsilon_k, \tilde{\varepsilon})$ is
singular. Since $\delta_k=\varepsilon_k/\lambda$, $|\delta_k| \leq
1$ and $\tilde{\delta}=\tilde{\varepsilon}/\lambda$,
$|\tilde{\delta}| \leq 1$,
$$I - B(\delta_k, \tilde{\delta}) \in
 \left\{I-\sum_{i=1}^K p_i\check{A}^{-1}A_i, [-\hat{p},\hat{p}]\right\}. $$
The latter means that a singular matrix belongs to a set of
nonsingular ones, which is a contradiction.

(vi)$\Rightarrow$(i) Assume that $\{A(p), \inum{p}\}$ is singular
and $\check{A}=A(\check{p})$ is nonsingular. By
(i)$\Leftrightarrow$(v)
\begin{equation}\label{6->1}
\det\left(I-\tilde{p}_k\check{A}^{-1}A_k - \sum_{i=1,i\neq k}^K
\varepsilon^{(k)}_i\hat{p}_i\check{A}^{-1}A_i\right)=0
\end{equation}
for some $k\in{\cal K}$, $\tilde{p}_k\in [-\hat{p}_k, \hat{p}_k]$,
$\varepsilon^{(k)}\in Q_{K-1}$. This implies that the maximum
magnitude real spectral radius of each matrix
$\tilde{p}_k\check{A}^{-1}A_k + \sum_{i=1,i\neq k}^K
\varepsilon^{(k)}_i\hat{p}_i\check{A}^{-1}A_i$, which satisfies
(\ref{6->1}), is equal to $1$. The latter contradicts to (vi).
\end{proof}

All the conditions of Theorem \ref{nscGen}, except (iv) which is
numerical equivalent of (ii) and (vi) which is numerical equivalent
of the corresponding parametric problem, are in terms of parametric
problems involving one interval parameter. By the following example
we illustrate one of the advantages of considering parametric
problems involving only one interval parameter (instead of more) in
proving regularity/singularity of an interval parametric matrix.
Some other advantages are discussed latter on.

\begin{example}\label{1param}
Consider the interval parametric matrix
$$\begin{pmatrix}
1 + p_1 & p_3 & -1\\  p_2 & 1 + p_1 & p_3\\
 -1 &  p_2 & p_1/3
\end{pmatrix}, \qquad \begin{matrix}p_1\in [-\frac{3}{4}, \frac{3}{4}],\\ p_2, p_3\in [-\frac{1}{2}, \frac{1}{2}].\end{matrix}
$$
The sufficient condition \cite{Popova04}
\begin{equation} \label{sufCond}
\rho\left(\sum_{i=1}^K \hat{p}_k|A^{-1}(\check{p})A_i|\right) < 1
\end{equation}
for regularity of an interval parametric matrix does not hold for
the considered interval parametric matrix, $\rho \approx 1.54$. The
sufficient conditions from \cite{Rump94} and \cite{Popova19} also
fail.

We apply this sufficient condition for testing regularity of the
interval parametric matrices in Theorem \ref{nscGen}-(ii) and prove
that for each $k=1,2,3$, and each $\varepsilon^{(k)}\in Q_2$ the
corresponding interval parametric matrix (involving only one
interval parameter) is regular, $\max_{k\in{\cal K},
\varepsilon^{(k)}\in Q_2}\{\rho_k\}\approx 0.841$. Thus, the
considered interval parametric matrix, involving three interval
parameters, is regular by Theorem \ref{nscGen}-(ii).
\end{example}

It is advantageous that the conditions of Theorem \ref{nscGen} are
read negated. Thus, when testing some of the equivalent conditions,
we either find a singular matrix within the interval parametric one,
or prove regularity of the latter. This is illustrated in the
numerical examples. This double property of the necessary and
sufficient regularity conditions of Theorem \ref{nscGen} is also
illustrated by the Algorithm \ref{algo} below.

A two-fold application of Theorem \ref{nscGen}-(iii), for verifying
nonsingularity of the parametric matrix and finding the exact
interval hull of a parametric solution set, is discussed in Section
\ref{Hull}.

We do not know any method providing bounds for the determinant of an
interval matrix involving general parameter dependencies. The
simplest way to check condition Theorem \ref{nscGen}-(v) is to
determine if each of the univariate polynomials
\begin{equation}\label{numDet}
\det\left(A(\check{p})+p_kA_k+\sum_{i=1,i\neq k}^K
\varepsilon_i\hat{p}_iA_i\right), \qquad k\in{\cal K}, \;
\varepsilon\in Q_{K-1},
\end{equation}
has real roots in the corresponding interval $[-\hat{p}_k,
\hat{p}_k]$. Furthermore, the real roots (if present) for a
polynomial can be isolated in polynomial time \cite{Basu}. Hence,
determining the real roots, if any, we obtain {\em singular} real
matrices within the interval parametric matrix.

Following the proof (vi)$\Leftrightarrow$(i) in Theorem
\ref{nscGen}, condition (vi) can be verified by finding all real
solutions of the following constrained polynomial equation
\begin{equation}\label{numRho}
\begin{split}
&\det\left(\lambda I - p_kA^{-1}(\check{p})A_k - \sum_{i=1,i\neq k}^K \varepsilon_i\hat{p}_iA^{-1}(\check{p})A_i\right) = 0,\\
&-\hat{p}_k \leq p_k \leq \hat{p}_k, \qquad 0\leq |\lambda| \leq 1
\end{split}
\end{equation}
for each $k\in{\cal K}$, $\varepsilon\in Q_{K-1}$. Alternatively,
one can use some interval method, if applicable, for bounding the
range of the real eigenvalues of the corresponding one-parameter
interval matrix in (vi), for example, \cite{Kolev10},
\cite{HlaDaney11}. However, it is proven in \cite{KreiLakRohnKa}
that bounding the eigenvalues is an NP-hard problem even in case of
a nonparametric interval matrix.

As in the nonparametric case \cite{RohnLAA89}, $\varrho_0$ in
Theorem \ref{nscGen}-(vi) cannot be replaced by $\rho$. This is
demonstrated below in Example \ref{Hudak}, which is parametric
version of Hudak's example \cite{RohnLAA89}, \cite{Hudak}.

\begin{example}\label{Hudak}
Consider the interval parametric matrix
$$\begin{pmatrix}
p_1 & -43 & 49\\ -31 & p_1 & -35 \\ 25 & -35 & p_2
\end{pmatrix}, \qquad \begin{matrix}p_1\in [31, 41],\\ p_2\in [28, 38].\end{matrix}
$$
The sufficient condition (\ref{sufCond}) for regularity of an
interval parametric matrix is not satisfied, $\rho \approx 1.72$.

For each $k=1,2$ and $\varepsilon\in\{-1,1\}$, the corresponding
polynomial (\ref{numDet}) does not have real roots in the parameter
interval, which means that Theorem \ref{nscGen}-(v) holds true and
the considered interval parametric matrix is regular.

For each $k=1,2$ and $\varepsilon\in\{-1,1\}$, solving the
corresponding  equation (\ref{numRho}), we find the real values of
$\varrho_0 = |\lambda|$, which lie in the interval $[-1,1]$. For
each $k=1,2$ and $\varepsilon\in\{-1,1\}$, the obtained biggest
$\varrho_0 = |\lambda|$ is $\approx 0.968164$. The latter also means
that Theorem \ref{nscGen}-(vi) holds true and the considered
interval parametric matrix is regular.
\end{example}

The proof (vi)$\Leftrightarrow$(i) in Theorem \ref{nscGen} shows
that if $\{A(p), \inum{p}\}$ is singular, then it contains real
singular matrices of special form. The proof also reveals a
methodology (\ref{numDet}) for finding real singular matrices within
an interval parametric matrix. Solving the problems (\ref{numDet})
is computationally simpler than solving the problems (\ref{numRho})
since the first problem involves only one interval parameter.

\begin{algorithm} \label{algo}
Finding singular matrices within an interval parametric matrix
$\left\{A(p),\inum{p}\right\}$ or proving regularity of the latter.
\medskip

\begin{enumerate}
\item {\bf For} $k=1,2,\ldots,K$; $q={\cal K}\setminus\{k\}$;

\item {\bf For} $\varepsilon^{(k)}\in Q_{K-1}$

\item Finding the real solutions of the constrained polynomial
equation
$$
\det\left(A(p_k, \varepsilon^{(k)}\circ\hat{p}_q)\right) = 0, \qquad
-\hat{p}_k\leq p_k\leq \hat{p}_k,
$$
where $A(p_k, \varepsilon^{(k)}\circ\hat{p}_q) := A(\check{p}) +
p_kA_k + \sum_{j\in q} \varepsilon^{(k)}\hat{p}_jA_j$.

{\bf Denote}
$$
{\cal L}(k,\varepsilon^{(k)}) := \left\{\{k, \varepsilon^{(k)},
\tilde{p}_k\} \mid \tilde{p}_k\in [-\hat{p}_k,\hat{p}_k],
\det(A(\tilde{p}_k, \varepsilon^{(k)}\circ\hat{p}_q))=0\right\}.
$$

\item  {\bf If } ${\cal L}(k,\varepsilon^{(k)})\neq\emptyset$, {\bf
then}\quad  Return ${\cal L}(k,\varepsilon^{(k)})$;

\hspace*{1.5in} {\bf Terminate: }$\left\{A(p),\inum{p}\right\}$ is
singular.

\item {\bf End } (For of $k$)
\item {\bf End } (For of $\varepsilon^{(k)}$)
\item {\bf Return }$\left\{A(p),\inum{p}\right\}$ is regular.
\end{enumerate}
\end{algorithm}

With obvious modifications, Algorithm \ref{algo} can find all real
singular matrices that correspond to the negation of Theorem
\ref{nscGen}-(v). Proving singularity is not a priory exponential.
One can apply some heuristics to start Algorithm \ref{algo} with a
parameter, which most likely will give a singular matrix.

The following example demonstrates that, in general, proving
regularity we cannot reduce neither the number $K$ of tested
parameters nor the number of $\varepsilon^{(k)}\in Q_{K-1}$.

\begin{example}\label{redNumb}
Consider the interval parametric matrix
$$A(p)=\begin{pmatrix}
\frac{6}{5} + p_1 & p_3 & -1\\ 2 + p_2 & \frac{6}{5} + p_1 & p_3\\
 -1 & 2 + p_2 & \frac{1}{3}p_1
\end{pmatrix}, \qquad p_i\in [-1, 1], \; i=1,2,3.
$$
For each $k=1,2,3$ and $\varepsilon^{(k)}\in Q_2$, the corresponding
polynomial  (\ref{numDet}) does not have real roots in the parameter
interval, except for $k=2$ and $\varepsilon^{(2)} = (1,-1)^\top$.
This means that Theorem \ref{nscGen}-(v) does not hold true and the
considered interval parametric matrix is singular.

In the exceptional case, ${\rm det}\left(A(p_2, \hat{p}_1,
-\hat{p}_3)\right) = -\frac{13}{25}-\frac{22}{15}p_2-p_2^2$ and we
obtain two values for $p_2\in [-1,1]$ which make the determinant
equal to zero. Thus, for $\tilde{p}'= (1, -13/15, -1)^\top$ and
$\tilde{p}''= (1, -3/5, -1)^\top$ we obtain explicitly two real
singular matrices contained in the considered interval parametric
matrix.
\end{example}

\section{Interval hull of a parametric solution set} \label{Hull}

If $A(\check{p})$ is nonsingular, for each $k\in{\cal K}$,
$\varepsilon\in Q_{K-1}$,
\begin{eqnarray*}
\Sigma\left(A(p_k,\varepsilon\circ\hat{p}_{\tilde{q}}),
b(p_k,\varepsilon\circ\hat{p}_{\tilde{q}}), [-\hat{p}_k,
\hat{p}_k]\right) & = &
x\left(p_k,\varepsilon\circ\hat{p}_{\tilde{q}}\right)\mid_{[-\hat{p}_k,
\hat{p}_k]} \\
& = &
\left.\left(A(p_k,\varepsilon\circ\hat{p}_{\tilde{q}})\right)^{-1}b(p_k,\varepsilon\circ\hat{p}_{\tilde{q}})\right|_{[-\hat{p}_k,
\hat{p}_k]}
\end{eqnarray*}
is a piece of a one-parameter curve in $\R^n$, restricted to $p_k\in
[-\hat{p}_k, \hat{p}_k]$, and presents either a piece of the
boundary of $\Sigma\left(A(p), b(p), \inum{p}\right)$ or a subset of
the latter solution set. By Theorem \ref{BIT08} and Theorem
\ref{nscGen}-(iii) we have the following proposition.

\begin{proposition}\label{prop1}
If $A(\check{p})\in\R^{n\times n}$ is nonsingular and $b(p)$ is a
vector in $\R^n$, then
\begin{equation*}
\square\Sigma\left(A(p), b(p), \inum{p}\right) = \bigcup_{k\in{\cal
K}}\bigcup_{\varepsilon\in Q_{K-1}} \square
\Sigma\left(A(p_k,\varepsilon\circ\hat{p}_{\tilde{q}}),
b(p_k,\varepsilon\circ\hat{p}_{\tilde{q}}), [-\hat{p}_k,
\hat{p}_k]\right),
\end{equation*}
where for $i=1,\ldots,n$ \begin{multline*} \square
\Sigma_i\left(A(p_k,\varepsilon\circ\hat{p}_{\tilde{q}}),
b(p_k,\varepsilon\circ\hat{p}_{\tilde{q}}), [-\hat{p}_k,
\hat{p}_k]\right) = \\ \left[\inf_{p_k\in [-\hat{p}_k, \hat{p}_k]}
x_i\left(p_k,\varepsilon\circ\hat{p}_{\tilde{q}}\right),
\sup_{p_k\in [-\hat{p}_k, \hat{p}_k]}
x_i\left(p_k,\varepsilon\circ\hat{p}_{\tilde{q}}\right)\right].
\end{multline*}
\end{proposition}

Specifically, by Theorem \ref{nscGen}-(iii),
$$\inf/\sup x_i\left(p_k,\varepsilon\circ\hat{p}_{\tilde{q}}\right) = -/+\infty $$
if and only if $\{A(p), \inum{p}\}$ is singular. Thus, applying the
above Proposition together with Theorem \ref{nscGen}-(iii), we can
either find the exact interval hull of an interval parametric
solution set or prove that the interval parametric matrix is
singular.

The traditional approach for determining the interval hull of an
interval parametric linear system is to find the analytic solution
of the parametric system and then to bound the ranges of the
solution components in the parameter intervals. The latter problem
may be very difficult in presence of several interval parameters.
The methodology based on Proposition \ref{prop1} replaces the second
step in the traditional approach by a set of $K2^{K-1}$ range
computation problems involving a single interval variable. Since
each component of
$x\left(p_k,\varepsilon\circ\hat{p}_{\tilde{q}}\right)$ is a
rational function of one variable $p_k$, the exact extrema of
$x_i\left(p_k,\varepsilon\circ\hat{p}_{\tilde{q}}\right)$ in
$[-\hat{p}_k, \hat{p}_k]$ can be easily found for rational data.
This approach is applicable even when interval software is not
present since some available software, e.g., {\em
Mathematica}$^\circledR$ solve the latter problem. In terms of
guaranteed floating point computations, the methodology based on
Proposition \ref{prop1} can be combined with guaranteed interval
methods (for example \cite{GarloffEtAl12}) to reduce the number of
interval variables, and thus to obtain
\begin{itemize}
\item[(i)] guaranteed tight enclosure of the hull in floating point,
\item[(ii)] expanded applicability to problems with
 many interval parameters and/or problems with large parameter
 intervals.
\end{itemize}

In general, the methodology of Proposition \ref{prop1} is
appropriate for real-life problems involving relatively small number
of interval parameters (the computing time can be reduced by
distributed computations) or when the analysis is performed offline.
The proposed methodology is the only option for some systems with
very large parameter intervals. The present author applies this
methodology for finding the exact interval hull of parametric
solution sets with nonlinear boundary when constructs benchmark
examples and when estimates the quality of newly designed numerical
methods. Some examples of real-life applications involve analysis of
manipulators in robotics \cite{Notash15}, \cite{NazariNotash16},
\cite{Notash16} and guaranteed parameter set estimation for
exponential sums \cite{Garloff05}, \cite{Garloff07}.

\section{Radius of regularity}\label{regRad}
Since every interval $\inum{p}=[p^-,p^+]$ can be represented as
$\inum{p}=\check{p}+[-\hat{p},\hat{p}]$, every interval parametric
matrix $A(p)$, $p\in\inum{p}\in\IR^K$, can be equivalently
represented as $A(\check{p}+p)$, $p\in [-\hat{p}, \hat{p}]$.

\begin{definition}\label{radReg}
For a square interval parametric  matrix $A(p)$,
$p\in\inum{p}\in\IR^K$, its \emph{regularity radius} is defined by
$$
r^*(A(p),\inum{p}) :=  \inf\{r\geq 0 \mid A(\check{p}+rp) \text{  is
singular for some } p\in\R^K, p\in [-\hat{p}, \hat{p}]\}.
$$
\end{definition}
Specifically, $r^*(A(p),\inum{p})=\infty$ if no real $r$ exists such
that  $\left\{A(\check{p}+rp)\right.,$ $\left.p\in [-\hat{p},
\hat{p}]\right\}$ is singular. If $r^*(A(p),\inum{p}) < \infty$,
then the infimum is achieved as minimum.

Definition \ref{radReg} is not restricted to interval parametric
matrices involving affine linear parameter dependencies. In what
follows, however, we present an explicit formula for the radius of
regularity of interval parametric matrix involving affine-linear
dependencies and some necessary and sufficient conditions for its
infinite value.

\begin{theorem}\label{1/rhoThm}
Let $A(p)$ involve only affine-linear dependencies on
$p\in\inum{p}\in\IR^K$ and $A(\check{p})$ be nonsingular. Then,
\begin{multline}\label{1/rho}
r^*(A(p),\inum{p}) = \\
1/\max\left\{\varrho_0\left(B(p_k,\varepsilon^{(k)})\right) \mid
k\in{\cal K}, p_k\in [-\hat{p}_k,\hat{p}_k], \varepsilon^{(k)}\in
Q_{K-1}\right\},
\end{multline}
where $B(p_k,\varepsilon^{(k)}) =
p_k\left(A(\check{p})\right)^{-1}A_k +\sum_{i=1,i\neq k}^K
\varepsilon^{(k)}_i\hat{p}_i\left(A(\check{p})\right)^{-1}A_i$.
\end{theorem}
\begin{proof}
Let $r^*(A(p),\inum{p})<\infty$. For a given $r\geq 0$, there exists
a singular matrix within the interval parametric matrix
$\left\{A(\check{p}+rp), p\in [-\hat{p}_k,\hat{p}_k]\right\}$, by
Theorem \ref{nscGen}-(vi), if and only if
$$\varrho_0\left(rB(p_k,\varepsilon^{(k)})\right)  \geq 1 $$
for some $k\in{\cal K}$, $p_k\in [-\hat{p}_k,\hat{p}_k]$,
$\varepsilon^{(k)}\in Q_{K-1}$. The latter means that
$$
r\max\left\{\varrho_0\left(B(p_k,\varepsilon^{(k)})\right) \mid
k\in{\cal K}, p_k\in [-\hat{p}_k,\hat{p}_k], \varepsilon^{(k)}\in
Q_{K-1}\right\} \geq 1.
$$
Hence the minimum value of $r$ is given by (\ref{1/rho}).

If $r^*(A(p),\inum{p})=\infty$, then  $A(\check{p}+rp)$ is
nonsingular for each  $r\in\R$, $r\geq 0$, $p\in [-\hat{p},
\hat{p}]$, according to Definition \ref{radReg}. By Theorem
\ref{nscGen}-(vi) this is equivalent to
$r\varrho_0\left(B(p_k,\varepsilon^{(k)})\right) < 1$ for each
$k\in{\cal K}$, $p_k\in [-\hat{p}_k,\hat{p}_k]$,
$\varepsilon^{(k)}\in Q_{K-1}$ and each $r\in\R$, $r\geq 0$. Hence
$\varrho_0\left(B(p_k,\varepsilon^{(k)})\right)=0$ for each
$k\in{\cal K}$, $p_k\in [-\hat{p}_k,\hat{p}_k]$,
$\varepsilon^{(k)}\in Q_{K-1}$ and the convention $1/0=\infty$
implies  (\ref{1/rho}).
\end{proof}

Radius of regularity for interval parametric matrices was first
defined in \cite{Kolev14} under the assumption that
$r^*(A(p),\inum{p}) <\infty$. In several works, e.g.,
\cite{Kolev'CTA12}, \cite{Kolev14}, L. Kolev develops and applies a
methodology for either finding the regularity radius of an interval
parametric matrix or providing bounds for the latter. Beside various
other requirements, this methodology also assumes
$r^*(A(p),\inum{p}) <\infty$. We are not informed about any other
work discussing a methodology for checking the condition
$r^*(A(p),\inum{p}) <\infty$, respectively the condition
$r^*(A(p),\inum{p})=\infty$.

In view of Definition \ref{radReg}, the regularity radius can
be defined equivalently as
\begin{multline}\label{defRegDet}
r^*(A(p),\inum{p}) := \\ \inf\{r\geq 0 \mid
\det\left(A(\check{p}+rp)\right)=0  \text{ for some }  p\in\R^K,
p\in [-\hat{p}, \hat{p}]\}.
\end{multline}
According to (\ref{defRegDet}), $r^*(A(p),\inum{p}) = \infty$ if
$\det\left(A(\check{p}+rp)\right)=0 $ does not have real solutions
for any $r\in\R$, $r\geq 0$, $p\in\R^K$, $p\in [-\hat{p}, \hat{p}]$.

Applying Theorem \ref{1/rhoThm} and Theorem \ref{nscGen}, we obtain
the following equivalent conditions for an infinite radius of
regularity.

\begin{corollary}\label{cor1}
An interval parametric matrix $\left\{A(p), p\in\inum{p}\right\}$
involving affine-linear dependencies, has infinite radius of
regularity $r^*$ if and only if
\begin{itemize}
\item[(i)] the constrained equation
$$\det\left(\check{A}- r\left(p_kA_k +\sum_{i=1,i\neq k}^K
\varepsilon^{(k)}_i\hat{p}_iA_i\right)\right)=0, \qquad p_k\in
[-\hat{p}_k, \hat{p}_k], r\geq 0,
$$
does not have real solutions for each $k\in{\cal K}$,
$\varepsilon^{(k)}\in Q_{K-1}$,

\item[(ii)] equivalently, nonsingular $\check{A}=A(\check{p})$ and
$$
\rho_0\left(B(p_k, \varepsilon^{(k)})\right)=0, \qquad B(p_k,
\varepsilon^{(k)}):= p_k\check{A}^{-1}A_k + \sum_{i=1, i\neq k}^K
\varepsilon^{(k)}_i\hat{p}_i\check{A}^{-1}A_i,
$$
for each $k\in{\cal K}$, $p_k\in [-\hat{p}_k, \hat{p}_k]$,
$\varepsilon^{(k)}\in Q_{K-1}$.
\end{itemize}
\end{corollary}

Checking Corollary \ref{cor1}-(i), the parametric determinants
involve two parameters compared to the procedure (\ref{numDet}),
respectively Algorithm \ref{algo}. Checking Corollary
\ref{cor1}-(ii), however, applies the same procedure (\ref{numRho})
with $|\lambda|\geq 0$.

Checking any condition of Corollary \ref{cor1}, and finding all real
roots of the corresponding determinant, whenever they exist, by
 Theorem \ref{1/rhoThm} we either determine the finite
regularity radius of an interval parametric matrix, or prove that
the regularity radius is infinite.

\begin{example} Determining the radius of regularity for
the interval parametric matrices from Examples \ref{1param},
\ref{Hudak} and \ref{redNumb}, we obtain with $6$ digits precision:
\begin{eqnarray*}
r^{*}(A(p), p\in\inum{p}, {\rm Example \ \ref{1param}}) &=& 1.08209 ,\\
r^{*}(A(p), p\in\inum{p}, {\rm Example \ \ref{Hudak}}) &=&  1.03289,\\
r^{*}(A(p), p\in\inum{p}, {\rm Example \ \ref{redNumb}}) &=&
0.996413.
\end{eqnarray*}
\end{example}

\section{Conclusion}

By Poljak and Rohn \cite{PolyakRohn}, see also \cite{KreiLakRohnKa},
checking regularity is an NP-hard problem for the nonparametric
interval matrices (note that a $n\times n$ nonparametric interval
matrix can be considered as an interval parametric  matrix involving
$n^2$ parameters). That is why, some polynomially computable
sufficient conditions are used for verifying regularity.  However,
the sufficient regularity conditions may approximate quite rough a
matrix regularity and, depending on the width of the parameter
intervals, may fail proving regularity.

All of the necessary and sufficient regularity conditions, presented
in Theorem \ref{nscGen},  exhibit exponential behavior. They employ
a finite set of test matrices which cardinality is $K2^{K-1}$, where
$K$ is the number of non-degenerate interval parameters. This makes
the proposed methodology efficient for large matrices provided the
number of interval parameters is small.

The methodology, presented here, for proving regularity/singularity
of an interval parametric matrix is particularly useful in case of
large parameter intervals for which all sufficient regularity
conditions fail. This is just opposite to the requirement for narrow
intervals of the most interval methods in order to provide a good
quality of the result or for success. Since in case of large
parameter intervals, the interval matrix is very close to a singular
one, the proposed methodology may come very quickly to a singular
matrix. Thus, the proposed methodology has exponential complexity
only in case of regular interval parametric matrices, and could be
qualified as not a priori exponential.

A key feature of the discussed methodology is that the original
problem involving $K$ interval parameters is transformed to
$K2^{K-1}$ regularity problems involving only one interval
parameter. On one side, as discussed in Example \ref{1param}, this
leads to an increased applicability of the easy verifiable
sufficient regularity conditions. On the other hand, the one
parameter problems, whose solving is required by the necessary and
sufficient regularity conditions, are polynomially solvable. Most of
the problems are solvable exactly in exact arithmetic. The most
easily implementable criterion is Theorem \ref{nscGen}-(v).
Furthermore, the $K2^{K-1}$ one parameter problems are independent
of each other and, therefore, can be checked on parallel processors.
The latter increases the applicability of the proposed necessary and
sufficient regularity conditions.

As already mentioned, proving regularity of an interval parametric
matrix implies various other properties of these matrices.

\section*{Acknowledgements}
This work is supported by the Grant No BG05M2OP001-1.001-0003,
financed by the Bulgarian Operational Programme ``Science and
Education for Smart Growth''  (2014-2020) and co-financed by the
European Union through the European structural and investment funds,
as well as,  by the National Scientific Program ``Information and
Communication Technologies for  a  Single Digital Market  in
Science, Education and  Security (ICTinSES)'', contract No
DO1-205/23.11.2018, financed by the Ministry of Education and
Science in Bulgaria.


\end{document}